\documentclass[11pt]{article}

\usepackage{amsmath,amsthm,fullpage,amssymb,graphicx,bm,enumerate,color}
\usepackage{fourier,eucal}
\usepackage[bbgreekl]{mathbbol}

\newcommand{\tst}{\textstyle}

\newcommand{\m}[1]{\mathcal{#1}}
\newcommand{\bb}[1]{\mathbb{#1}}
\newcommand{\ur}[1]{\mathrm{#1}}
\newcommand{\s}[1]{\mathsf{#1}}

\newcommand{\eps}{\varepsilon}
\newcommand{\R}{\bb R}
\newcommand{\N}{\bb N}
\newcommand{\wc}{\rightharpoonup}

\renewcommand{\d}{\ur{d}}
\newcommand{\p}{\partial}

\renewcommand{\theta}{{\vartheta}}
\renewcommand{\O}{{\Omega}}
\newcommand{\oO}{\overline{\O}}
\newcommand{\pO}{{\p \O}}
\newcommand{\oh}{\frac{1}{2}}

\newcommand{\la}{\langle}
\newcommand{\ra}{\rangle}

\newcommand{\zu}{{\tilde{u}}}
\newcommand{\zv}{{\tilde{v}}}
\newcommand{\zp}{{\tilde{\phi}}}
\newcommand{\zvp}{{\tilde{\varphi}}}

\newcommand{\W}{{W}}

\newcommand{\Sz}{{\rm (S)$_0$}}
\newcommand{\hu}{{h}}

\newtheorem{theorem}{Theorem}
\newtheorem{corollary}{Corollary}
\newtheorem{lemma}{Lemma}
\newtheorem{remark}{Remark}

\begin{document}

\title{Finite Element Convergence for the Joule Heating Problem with Mixed Boundary Conditions}
\author{Max Jensen\footnote{Department of Mathematical Sciences, University of Durham} \and Axel M\aa lqvist\footnote{Department of Information Technology, Uppsala University}}

\maketitle

\begin{abstract}
We prove strong convergence of conforming finite element approximations to the stationary Joule heating problem with mixed boundary conditions on Lipschitz domains in three spatial dimensions. We show optimal global regularity estimates on creased domains and prove {\em a priori} and {\em a posteriori} bounds for shape regular meshes.
\end{abstract}

\hspace{4mm}


\hspace{4mm}


% ------------------------------------------------------------------------------------------------------------------------
\section{Introduction}
% ------------------------------------------------------------------------------------------------------------------------

The stationary Joule heating problem is a two way coupled system of non-linear elliptic partial differential equations modelling the heat and electrical potential in a body. The electrical current acts as a heat source in a resistive material while the temperature feeds back to the electrical potential through the electrical conductivity. Joule heating is important in many micro-electromechanical systems, where the effect is used to achieve very exact positioning at the micro scale, e.g.~\cite{HeTi06}. The Joule heating problem is also studied for the design of semiconductors, in particular in the setting of thermistors. In applications boundary conditions of mixed type are typically used.

The main difficulty in proving the existence of finite energy solutions to the Joule heating problem is that, given a finite energy potential, the source term of the heat equation is in general only in $L^1$, which means that the usual variational framework is not directly available. This issue has been studied in \cite{CI88,GaHe94,BoOr08}, for Dirichlet boundary conditions, and later in \cite{CI89,HRS93}, for mixed boundary conditions. Multiplicity of solutions and stability was studied in \cite{CI02}. Similar questions have also been raised for the time dependent case, see for instance~\cite{AC94,YL94,KSF09}.

There have been several works on the numerical solution of the Joule heating problems in recent years. For the steady state formulation both conforming and non-conforming finite element methods have been studied using homogeneous Dirichlet boundary conditions \cite{ZL05,ZYL11} and homogeneous mixed boundary conditions \cite{ZH10}. Under assumption of sufficient regularity of the solution and small data, {\em a priori} error bounds have been derived with convergence rates. There have also been parallel investigations into numerical methods for the time dependent Joule heating problem, see e.g.~\cite{ElLa95,AkLa05}. The assumption on small data can here be avoided since the Gr\"{o}nwall lemma is available. An {\em a posteriori} error bound for a time dependent obstacle thermistor problem is presented in~\cite{AY06}.

In this paper we prove the strong convergence (of subsequences in case of non-unique exact solutions) of Galerkin approximations to finite energy solutions of the Joule heating problem in three dimensions with mixed boundary conditions, using only very mild assumptions on the computational domain and the data. The analysis covers, in particular, conforming finite element approximations. To this end we introduce a truncation operator in the approximate potential without affecting the consistency of the method. Thereby we avoid the assumption of $L^\infty$ bounds on the discrete potential solution, independent of the mesh size, which are used in \cite{HoLa10}. These $L^\infty$ bounds are very difficult to realize in practice in three spatial dimensions. They also impose restrictions on the computational meshes as well as the order of convergence of the method. Under the assumption of a so-called {\em creased} domain together with a sufficiently weak temperature dependency in the electrical conductivity we also prove optimal global regularity estimates together with local estimates guaranteeing smooth solutions away from the boundary given smooth data. We further prove {\em a priori} and {\em a posteriori} error bounds for conforming finite element approximations on shape regular meshes. In our analysis the small data assumption relaxes as the coupling of the equations weakens.

The paper is organized as follows. In Section 2 we introduce the stationary Joule heating problem with mixed boundary conditions. In Section 3 we study the continuity properties of the differential operator to show the convergence of Galerkin approximations to finite energy solutions. In Section 4 study the global and interior regularity of solutions on creased domains. Finally, in Section 5 we derive optimal {\em a priori} and {\em a posteriori} error bounds for $\s h$-adaptive conforming finite element approximations to the Joule heating problem.

% ------------------------------------------------------------------------------------------------------------------------
\section{The Statement of the Stationary Problem}
% ------------------------------------------------------------------------------------------------------------------------

Let $\O$ be a bounded Lipschitz domain in $\R^3$. Let $D_\phi$ and $D_u$ be subsets of $\pO$, whose boundaries $\p D_\phi$ and $\p D_u$ are Lipschitz regular manifolds and set $N_\phi := \pO \setminus \overline{D_\phi}$ and $R_u := \pO \setminus \overline{D_u}$. We shall impose Dirichlet boundary conditions for $\phi$ and $u$ on $D_\phi$ and $D_u$, Neumann conditions for $\phi$ on $N_\phi$ and Robin conditions for $u$ on $R_u$.

The scale of Sobolev spaces is denoted by $\W_s^p$. Let, for $s > 1/p$,
\[
\W_s^p(\O;D_\phi) := \{ v \in \W_s^p(\O) : v|_{D_\phi} = 0 \}.
\]
Define $\W_s^p(\O;D_u)$ analogously and denote $\W_1^2$ spaces also with $H^1$.

Assume for the Dirichlet data that $g_\phi \in W_1^3(\oO)\cap L^\infty(\oO)$ and $g_u \in \W_{1/2}^2(\pO)$ and for the Robin data that $\hu \in \W_{-1/2}^2(\pO)$. Let $\sigma \in C^1(\R)$ be bounded from below by a positive $\sigma_\circ \in \R$ and from above by $\sigma^\circ \in \R$ and let $\kappa \in L^\infty(R_u)$ be non-negative. Assume that there are the Poincar\'e-Friedrichs inequalities
\begin{align}\label{eq:pfi} \begin{array}{rll}
\| \psi \|_{L^2(\O)} \lesssim & \| \nabla \psi \|_{L^2(\O, \R^3)} & \qquad \forall \psi \in \W_1^2(\O;D_\phi),\\[2mm]
\| w \|_{L^2(\O)}    \lesssim & \| \nabla w \|_{L^2(\O, \R^3)} + \| \sqrt{\kappa} \, w \|_{L^2(R_u)} & \qquad \forall w \in \W_1^2(\O;D_u).
\end{array} \end{align}
Allow $D_u = \emptyset$ provided the Poincar\'e-Friedrichs inequality remains valid.

The strong formulation of the Joule heating problem is to find $\phi \in \W_2^\infty(\O)$ and $u \in \W_2^\infty(\O)$ such that
\begin{align} \label{eq:sphi}
- \nabla\cdot(\sigma(u) \nabla \phi) & = 0 \qquad \qquad \Leftrightarrow \qquad \qquad - \sigma(u) \Delta u = \nabla \sigma(u) \cdot \nabla \phi,\\ \nonumber
- \Delta u - \sigma(u) |\nabla \phi |^2 & = 0
\end{align}
with the Dirichlet conditions $\phi|_{D_\phi} = g_\phi$, $u|_{D_u} = g_u$ and the natural boundary conditions \mbox{$\p_\nu \phi = 0$} on $N_\phi$ and $\kappa u + \p_\nu u = \hu$ on $R_u$ with the outward unit normal $\nu$.

\begin{remark}
For some applications the Lipschitz assumption on $\O$ is too restrictive, a good examples being geometries which locally resemble the two-brick domain. We point out that Theorem \ref{thm:exis} remains valid for domains for which the usual Sobolev embedding holds, a trace operator is available and integration-by-parts can be carried out. For example, see \cite{JW84} for more information in this direction. For Theorem \ref{thm:reg}, however, the Lipschitz assumption is an essential part of the definition of creased domains.
\end{remark}

% ------------------------------------------------------------------------------------------------------------------------
\subsection{The Weak Formulation of the Stationary Problem}
% ------------------------------------------------------------------------------------------------------------------------

A weak solution of the stationary Joule heating problem is a
\[
(\phi, u) = (g_\phi + \zp, g_u + \zu) \in H^1(\O) \times H^1(\O)
\]
such that $\zp\in H^1(\O,D_\phi)$, $\zu \in H^1(\O,D_u)$ and
\begin{align} \label{eq:ws} \left. \begin{array}{cccclccl}
\la \sigma(u) \nabla \phi, \nabla \psi \ra & & & = & 0,\\[2mm]
\la \nabla u, \nabla w \ra & + & \la \kappa \, u, w \ra_{R_u} & = & \la \sigma(u) \nabla \phi \cdot \nabla \phi, w \ra & + & \la \hu, w \ra_{R_u}
\end{array} \quad \right\} \end{align}
for all $\psi \in H^1(\O;D_\phi)$ and $\forall \, w \in \W_1^\infty(\O;D_u)$. Indeed, the choice of spaces ensures that $\sigma(u) \nabla \phi \cdot \nabla \phi \in L^1(\O)$ which guarantees that the second equation is meaningful for all $w \in \W_1^\infty(\O;D_u)$.

\begin{lemma} \label{thm:max}
If $(\phi, u)$ is a solution of \eqref{eq:ws} then
\[
g_\circ \le \phi \le g^\circ, \qquad \mbox{ with } \qquad g^\circ =\max_{x \in D_\phi} g_\phi, \; g_\circ=\min_{x \in D_\phi} g_\phi.
\]
\end{lemma}

\begin{proof}
Define $\chi=\max(0,\phi-g^\circ)\in H^1(\O;D_\phi)$. One can use $\chi$ as a test function in equation (\ref{eq:ws}):
\begin{align*}
0=\la \sigma(u) \nabla \phi, \nabla \chi \ra & = \la \sigma(u) \nabla (\phi-g^\circ), \nabla \chi \ra\\[1mm]
& = {\textstyle \int_{{\rm supp}(\chi) \cap \O} \sigma(u) \nabla \chi \cdot \nabla \chi \, \d x} = \la \sigma(u) \nabla \chi, \nabla \chi \ra.
\end{align*}
Now use the Poincar\'e-Friedrichs inequality to get $\|\chi\|_{L^2(\O)}=0$, so $\phi\leq g^0$. An analogous argument with $g_\circ$ gives $\phi \geq g_\circ$. 
\end{proof}

Because of the maximum principle we may introduce an equivalent weak formulation which employs the cut-off functional
\[
\lceil f \rceil := \min(\max(f+g_\phi, g_\circ), g^\circ)-g_\phi.
\]
Then $g_\circ - g_\phi \le \lceil f \rceil \le g^\circ - g_\phi$ and $\lceil \zp \rceil = \zp$. This functional is essential in the proof of the convergence of Galerkin solutions without the need for a discrete maximum principle; a property desirable from the numerical point of view.

\begin{lemma}
The set of functions which satisfy
\begin{align} \label{eq:wse} \left. \begin{array}{ccl}
\la \sigma(u) \nabla \phi, \nabla \psi \ra & = & 0,\\[2mm]
\la \nabla u, \nabla w \ra + \la \kappa \, u, w \ra_{R_u} & = &\\[2mm]
&& \hspace{-25mm} - \la \sigma(u) \, \lceil \zp \rceil \nabla \phi, \nabla w \ra + \la \sigma(u) \nabla g_\phi \cdot \nabla \phi, w \ra + \la \hu, w \ra_{R_u}
\end{array} \quad \right\} \end{align}
for all $(\psi,w) \in H^1(\O;D_\phi) \times H^1(\O;D_u)$ is equal to the set of solutions of~\eqref{eq:ws}.
\end{lemma}

\begin{proof}
The identity
\[
\la \sigma(u) \nabla \phi \cdot \nabla \phi, w \ra = - \la \sigma(u) \, \zp \, \nabla \phi, \nabla w \ra + \la \sigma(u) \nabla g_\phi \cdot \nabla \phi, w \ra
\]
follows from Lemma 1 in \cite{HRS93}. The cut-off functional may be used because of Lemma \ref{thm:max} above. The larger space of test functions does not change the set of weak solutions due to density and does not lead to infinite terms in~\eqref{eq:wse}. 
\end{proof}

We define the space $X := H^1(\O;D_\phi) \times H^1(\O;D_u)$ and the {\em affine} mapping
\begin{align*}
L : \; & X \to X^*, \; (\zvp,\zv) \mapsto \left( (\psi,w) \mapsto
 \begin{pmatrix}
 \la \sigma_\circ \nabla \varphi, \nabla \psi \ra & & \\
 \la \nabla v, \nabla w \ra & + & \la \kappa \, v, w \ra_{R_u}
 \end{pmatrix} \right)
\end{align*}
and the nonlinear mapping
\begin{align*}
N : X \to X^*, (\zvp,\zv) \mapsto \left( (\psi,w) \mapsto  \begin{pmatrix} \la (\sigma(v) - \sigma_\circ) \nabla \varphi, \nabla \psi \ra\\ \la \sigma(v) \, \lceil \zvp \rceil \nabla \varphi, \nabla w \ra - \la \sigma(v) \nabla g_\phi \cdot \nabla \varphi, w \ra \end{pmatrix} \right)
\end{align*}
and the functional
\[
b : X \to \R, \; (\psi,w) \mapsto \begin{pmatrix}
 0\\
 \la \hu, w \ra_{R_u}
 \end{pmatrix}
\]
where $\varphi = g_\phi + \zvp$ and $v = g_u + \zv$. Then equation \eqref{eq:wse} is in operator form
\begin{align} \label{eq:wso}
L x + N x = b
\end{align}
with $x = (\zp, \zu) \in X$.

For $(\zvp,\zv), (\psi,w) \in X$ one has
\begin{align} \label{eq:N_bounded} \hspace*{-3mm} \left. \begin{array}{rl}
& \, N( (\zvp,\zv), (\psi, w))\\[1mm]
\le & \, \sigma^\circ \cdot \bigl( \| \nabla \zvp \|_{L^2(\O)} + \| \nabla g_\phi \|_{L^2(\O)} \bigr)\\[1mm]
& \, \phantom{\sigma^\circ} \cdot \bigl( \| \nabla \psi \|_{L^2(\O)} + (g^\circ - g_\circ) \| \nabla w \|_{L^2(\O)} + \| \nabla g_\phi \|_{L^3(\O)} \| w \|_{L^6(\O)} \bigr)\\[1mm]
\lesssim & \, \bigl( 1 + \| \nabla \zvp \|_{L^2(\O, \R^3)} \bigr) \, \| (\psi,w) \|_{X},
\end{array} \; \right\} \end{align}
where we use $\|(\cdot,\cdot)\|_X$ to denote the natural norm in the product space $X$, in this case
\[
\|(\psi,w)\|^2_X=\|\psi\|^2_{H^1(\O)}+\|w\|^2_{H^1(\O)}.
\]

Throughout the text we adopt the notational convention that for a function~$\flat$ one understands $\tilde{\flat} = \flat - g_\phi$ if $\flat$ is a Greek letter and $\tilde{\flat} = \flat - g_u$ if $\flat$ is a Latin letter. We call $H^1(\O;D_\phi)$ the first and $H^1(\O;D_u)$ the second component of $X$. In this spirit we also refer, for example, to $\la \sigma_\circ \nabla \varphi, \nabla \psi \ra$ as the first component of $L$. Furthermore, we distinguish between $\phi$, which is a solution, and $\varphi$, which is a generic trial function.

% ------------------------------------------------------------------------------------------------------------------------
\section{Existence and Convergence of Galerkin Approximations}
% ------------------------------------------------------------------------------------------------------------------------

Consider a hierarchical family of subspaces $\{ X_n \}_{n \in \N} = \{ P_n \times U_n \}_{n \in \N}$ whose union is dense in $X$. A Galerkin solution $x_n \in X_n$ of \eqref{eq:wso} is a solution of
\begin{align} \label{eq:wsg}
\la L x_n + N x_n, y \ra = \la b, y \ra, \qquad \forall \, y \in X_n.
\end{align}
Lemma \ref{thm:demi} examines continuity properties of $L$ and $N$.

\begin{lemma} \label{thm:demi}
Let $\{ y_n \}_n = \{ (\zvp_n, \zv_n) \}_n$ be a sequence in $X$ and $y = (\zvp, \zv) \in X$ such that $\zvp_n \to \zvp, \zv_n \wc \zv$ as $n \to \infty$. Then $L y_n \wc L y$ weakly and $N y_n \to N y$ strongly in $X^*$.
\end{lemma}

\begin{proof}
Suppose there is a subsequence $\{ v_{n(k)} \}_k$ and an $\eps > 0$ such that
\begin{align} \label{eq:pwl}
\| \sigma(v_{n(k)}) \nabla \varphi_{n(k)} - \sigma(v) \nabla \varphi \|_{L^2(\O, \R^3)} > \eps \qquad \forall k \in \N.
\end{align}
The compactness of the embedding $H^1(\O) \hookrightarrow L^2(\O)$ and a corollary of the Riesz-Fischer theorem \cite[p.~161]{JJ99} imply that there is a subsequence, also denoted $\{ v_{n(k)} \}_k$, which converges pointwise almost everywhere. By possibly passing to another subsequence of indices we may also assume that $\{ \nabla \varphi_{n(k)} \}_k$ converges pointwise almost everywhere. The sequence
\[
\{ \sigma(v_{n(k)})^2\; \nabla \varphi_{n(k)} \cdot \nabla \varphi_{n(k)} \}_k
\]
is bounded in each component by $(\sigma^\circ)^2 |\nabla \varphi_{n(k)}|^2$. From the dominated convergence theorem, in the form of (Royden, p.~270), it follows that the sequence $\{ \sigma(v_{n(k)}) \nabla \varphi_{n(k)} \}_k$ converges strongly in $L^2(\O, \R^3)$. Observe that almost everywhere the poinwise limit of $\{ \sigma(v_{n(k)}) \nabla \varphi_{n(k)} \}_k$ is $\sigma(v) \nabla \varphi$, contradicting \eqref{eq:pwl}. Therefore $L y_n + N y_n$ converges, indeed strongly, in the first component. It also follows that the terms
\begin{align*}
\sigma(v_n) \, \lceil \zvp_n \rceil \nabla \varphi_n & \in L^2(\O, \R^3),\\
\sigma(v_n) \nabla g_\phi \cdot \nabla \varphi_n & \in L^2(\O),\\
\kappa v_n & \in L^2(\pO)
\end{align*}
converge strongly as $n \to \infty$. Hence $\{ L y_n \}_n$ converges weakly and $\{ N y_n \}_n$ strongly to $L y$ and $N y$ in $X^*$, respectively:
\[
\lim_n \sup_{z \neq 0} \frac{\la N y_n - N y, z\ra}{\| z \|_X} = 0, \qquad
\forall z \in X^* : \lim_n \la L y_n - L y, z\ra = 0,
\]
completing the proof. 
\end{proof}

The following lemma establishes a property of $L + N$ which is a variation of condition \Sz; a concept introduced by Browder, see \cite{B68} or \cite[IIB, p.583]{Zeidler}.

\begin{lemma} \label{thm:Sz}
Let $y_n = \{ (\zvp_n, \zv_n) \}_n$ be a sequence in $X$ and $y = (\zvp, \zv) \in X$ such that
\begin{align}
y_n & \wc y, \label{eq:w_conv}\\
L y_n + N y_n & \wc b, \label{eq:rhs_conv}\\
\lim_n \la L y_n + N y_n, (\zvp_n, 0) \ra & = \la b, (\zvp, 0) \ra, \label{eq:prod_convp}\\
\lim_n \la L y_n + N y_n, (0, \zv_n) \ra & = \la b, (0,\zv) \ra. \label{eq:prod_convu}
\end{align}
Then $y_n \to y$ strongly.
\end{lemma}

\begin{proof}
Adapting the argument of the proof of the previous lemma it follows analogously that $\{ \sigma(v_n) \nabla \psi \}_n$ converges strongly in $L^2(\O, \R^3)$. Using the strong convergence in $(*)$, one obtains
\begin{align*} \begin{array}{rcl}
0 & \le & \limsup_n \la \sigma(v_n) \nabla (\zvp - \zvp_n), \nabla (\zvp - \zvp_n) \ra\\[1mm]
& = & \limsup_n \bigl( \la \sigma(v_n) \nabla \zvp\phantom{_n}, \nabla  \zvp \ra - 2 \la \sigma(v_n) \nabla \zvp_n, \nabla \zvp \ra + \la \sigma(v_n) \nabla \zvp_n, \nabla \zvp_n \ra \bigr)\\
& \stackrel{(*)}{=} & \limsup_n \bigl(\la \sigma(v_n) \nabla \zvp_n, \nabla \zvp \ra - 2 \la \sigma(v_n) \nabla \zvp_n, \nabla \zvp \ra + \la \sigma(v_n) \nabla \zvp_n, \nabla \zvp_n \ra \bigr)\\
& \stackrel{\eqref{eq:prod_convp}}{=} & \limsup_n \bigl(- \la \sigma(v_n) \nabla \varphi_n, \nabla \zvp \ra + \la b, (\zvp,0) \ra + \la \sigma(v_n) \nabla g_\phi, \nabla ( \zvp_n - \zvp ) \ra  \bigr)\\
& \stackrel{\eqref{eq:rhs_conv}}{=} & - \la b, (\zvp,0) \ra + \la b, (\zvp,0) \ra = 0.
\end{array} \end{align*}
Therefore $\varphi_n$ converges strongly and Lemma \ref{thm:demi} becomes available. Hence $\la N y_n, y_n \ra \to \la N y, y \ra$ and
\begin{align*} \begin{array}{ccll}
\la b, y \ra &  \stackrel{\eqref{eq:rhs_conv}}{=} & \lim_n \la L y_n + N y_n, y \ra & = \lim_n \la L y_n, y\phantom{_n} \ra + \la N y, y \ra,\\
\la b, y \ra & \stackrel{\eqref{eq:prod_convp}, \eqref{eq:prod_convu}}{=} & \lim_n \la L y_n + N y_n, y_n \ra & = \lim_n \la L y_n, y_n \ra + \la N y, y \ra.
\end{array} \end{align*}
The weak continuity of $L$ implies that $\lim_n \la L y_n, y \ra = \la L y, y \ra$, cf.~\cite[p.~422]{DuSch88}. Therefore
\begin{align*}
0 = & \lim_n \la L y_n, y_n - y \ra = \lim_n \la L y_n, y \ra - 2 \la L y_n, y \ra + \la L y_n, y_n \ra\\
= & \lim_n \la L y, y \ra - 2 \la L y_n, y \ra + \la L y_n, y_n \ra = \lim_n \la L (y - y_n), y - y_n \ra.
\end{align*}
It follows from the coercivity of the linear part of $L$ that $y_n \to y$ in~$X$. 
\end{proof}

Let $T_n: X_n\rightarrow X_n$ be defined by $y_n=T \hat{y}_n$, where $y_n = (\zvp_n, \zv_n) \in X_n$ is given as the solution to
\begin{align} \label{eq:T_def}
\la L y_n + N (\zvp_n, \hat{v}_n), (\psi, w) \ra = \la b, (\psi, w) \ra, \qquad (\psi, w) \in X_n
\end{align}
with $\hat{y}_n = (\hat{\varphi}_n, \hat{v}_n)$. Algorithmically an iteration with $T$ corresponds to a method with the primary variable $\hat{v}_n$ and the dummy variable $\hat{\varphi}_n$ as $\hat{\varphi}_n$ does not explicitly appear in the next step of the iteration.

\begin{lemma} \label{thm:bound}
There exists a radius $r$, independent of $n$, such that the range of $T_n$ belongs to
\[
B_r := \{ y \in X : \| y \|_X \le r \}
\]
for all $n \in \N$.
\end{lemma}

\begin{proof}
Let $y_n = (\zvp_n, \zv_n) = T_n \hat{y} = T_n (\hat{\varphi}_n, \hat{v}_n)$. The first component of \eqref{eq:T_def} gives, with $\psi=\zvp_n$, the identity $0 = \la \sigma(\hat{v}_n) \nabla (\zvp_n + g_\phi), \nabla \zvp_n \ra$. Thus, with the above Poincar\'{e}-Friedrichs inequality for $H^1(\O;D_\phi)$:
\begin{equation*}
\| \zvp_n \|^2_{H^1(\O)} \lesssim \la \sigma(\hat{v}_n) \nabla \zvp_n, \nabla \zvp_n \ra = - \la \sigma(\hat{v}_n) \nabla g_\phi, \nabla \zvp_n \ra.
\end{equation*}
The Cauchy-Schwarz inequality now gives $\| \zvp_n \|_{H^1(\O)} \lesssim \|\nabla g_\phi\|_{L^2(\O, \R^3)}$. Recall \eqref{eq:N_bounded} with $(\psi,w) = (\zvp_n, \zv_n)$ and $(\zvp, \zv) = (\zvp_n, \hat{v}_n)$. The Cauchy-Schwarz inequality, \eqref{eq:pfi} and the coercivity of the linear part of $L$ give the boundedness of $\zv_n$. 
\end{proof}

Observe that the fixed points of $T_n$ are exactly the Galerkin solutions in the sense of \eqref{eq:wsg}.

\begin{lemma}
The mapping $T_n$ has at least one fixed point $x_n$.
\end{lemma}
\begin{proof}
We have that $T_n: B_r \cap X_n \rightarrow B_r \cap X_n$. The First Lemma of Strang (Braess, 2007, p.106) implies that the Galerkin solution of a linear elliptic equation changes continuously in the $H^1$-norm as the diffusion coefficient is varied in the $L^\infty$-norm. Therefore, looking at the first component in \eqref{eq:T_def}, $\zvp_n$ depends continuously on $\hat{v}_n$, taking the equivalence of norms in the finite-dimensional $X_n$ into account. With $\zvp_n$ determined, $\zv_n$ can be computed from $(\zvp_n, \zv_n) = L^{-1} (b - N (\zvp_n, \hat{v}_n))$. Lemma \ref{thm:demi} showed a sequential continuity of property of $L$ and $N$ which guarantees that the finite-dimensional Galerkin restrictions $X_n \to X_n$ are continuous. Equally the Galerkin restriction $X_n \to X_n$ of the affine mapping $L^{-1}$ is continuous. This means that $T_n$ is a continuous map $T_n: B_r \cap X_n \rightarrow B_r \cap X_n$, so Brouwer's fixed point theorem gives the existence of a fixed point~$x_n$. 
\end{proof}

It is a direct consequence that Galerkin solutions exist for all $n \in \N$ and (that at least one of them) are contained in $B_r$. The next theorem conceptually builds upon Proposition 27.4 in \cite[vol.~II B]{Zeidler} where Lemma \ref{thm:Sz} is replaced by \Sz.

\begin{theorem} \label{thm:exis}
There exists a subsequence of Galerkin solutions $\{ x_{n(k)} \}_k = \{ (\zvp_{n(k)}, \zv_{n(k)}) \}_k$ and an $x = (\zvp, \zv)$ in $X$ such that $x_{n(k)} \to x$ strongly in $X$ and $x$ solves \eqref{eq:wso}. If the solution $x$ of \eqref{eq:wso} is unique then the whole sequence converges.
\end{theorem}

\begin{proof}
It follows from \eqref{eq:wsg} that $\la (L+N) x_n, y \ra \to \la b , y \ra$ for all fixed $y \in X_m$, $m \in \N$. Observe that $L + N$ is a bounded operator, see \eqref{eq:N_bounded} for $N$. Thus with $\{ x_n \}_n$ also the sequence $\{ (L+N) x_n \}_n$ is bounded. Consequently, $(L+N) x_n \wc b$ in $X^*$ as $n \to \infty$, see Proposition 21.26 (c),(f) in \cite[vol.~II A]{Zeidler}. The sequence $\{ x_n \}_n$ is bounded in the reflexive Banach space $X$. Thus there exists an $x \in X$ and a subsequence $\{ x_{n(k)} \}_k$ with $x_{n(k)} \wc x$ as $k \to \infty$. It follows from \eqref{eq:wsg} that
\begin{align*}
\lim_k \la L x_{n(k)} + N x_{n(k)}, (\zvp_{n(k)}, 0) \ra & = \lim_k \la b, (\zvp_{n(k)}, 0) \ra = \la b, (\zvp, 0) \ra,\\
\lim_k \la L x_{n(k)} + N x_{n(k)}, (0, \zv_{n(k)}) \ra  & = \lim_k \la b, (0, \zv_{n(k)}) \ra  = \la b, (0,\zv) \ra.
\end{align*}
Therefore, it follows from Lemma \ref{thm:Sz} that $x_{n(k)} \to x$ as $k \to \infty$. The continuity established in Lemma \ref{thm:demi} yields $(L + N) x_{n(k)} \wc (L + N) x$ as $k \to \infty$ and therefore $(L + N) x = b$. If the exact solution is unique, then the Galerkin approximations can only have one accumulation point. 
\end{proof}

We remark that this convergence result applies, in particular, to conforming finite element methods.

% ------------------------------------------------------------------------------------------------------------------------
\section{Regularity}
% ------------------------------------------------------------------------------------------------------------------------

In this section we investigate how the regularity estimates for the Poisson problem with mixed boundary conditions, derived in \cite{MM07}, carry over to the Joule heating problem. These bounds are sharp in the Poisson setting. In general they are also sharp for the Joule problem---noting that equation \eqref{eq:sphi} takes the form of Poisson's equation when choosing a constant $\sigma$.

The underlying question is whether additional regularity can be gained if the type of the boundary condtions only changes at re-entrant corners:
\begin{quote}
{\bf (C)} \; $\Omega$ is a creased domain.
\end{quote}
For the full definition of creased domains we refer to \cite{MM07}; here we only highlight that in the setting of the Joule heating problem the key conditions are that $\O$ is a bounded Lipschitz domain, that $D_\phi$ and $D_u$ are open and non-empty and that $\partial D_\phi$ and $\partial D_u$ are not re-entrant, meaning that the angles between $D_\phi$ and $N_\phi$ as well as between $D_u$ and $R_u$ are strictly less than $\pi$.

Let $\m H_\eps \subset \R^2$ be the interior of the polygon with the vertices
\[
(0,0), \quad (\eps,0), \quad\left(1, \oh - \eps \right), \quad (1,1), \quad (1-\eps,1), \quad \left( 0, \oh + \eps \right).
\]

In the statement of the following theorem we let $D \in \{ D_\phi, D_u \}$ and $R \in \{ N_\phi, R_u \}$. Also $1/p + 1/p' = 1$ and $B^{p,q}_s$ denotes the scale of Besov spaces.

\begin{lemma}[\cite{MM07}] \label{thm:mitrea}
There exists an $\eps = \eps(\pO, D, R)$ in $(0,1/2)$ such that Poisson's equation is well-posed in the spaces
\begin{align} \label{eq:nlin} \left. \begin{array}{rlrl}
v & \in \W_{s+\frac{1}{p}}^p(\O), & \qquad \Delta v & \in \left( \W_{2 - s - \frac{1}{p}}^{p'}(\O; D) \right)^*,\\[3mm]
v|_D & = g \in B^{p,p}_s(D), & \qquad \p_\nu v|_R & = h \in B^{p,p}_{s-1}(R),
\end{array} \quad \right\} \end{align}
whenever $(s,1/p) \in \m H_\eps$.
\end{lemma}

Consequently we assume for the remainder of this section that $g_\phi \in B_s^{p,p}(D_\phi)$, $g_u \in B_s^{p,p}(D_u)$ and $h \in B_{s-1}^{p,p}(R_u)$. We also choose $$\eps := \min \{ \eps(\pO, D_\phi, N_\phi), \eps(\pO, D_u, R_u) \} \in (0 , 1/2).$$ We aim to prove existence of solutions with the smoothness and Lebesgue indices
\[
s = \frac{2 - \eps}{2}, \qquad t = \frac{8 + \eps}{12}, \qquad p = \frac{2}{1 - \eps}, \qquad q = \frac{12}{4 - \eps}.
\]
Then $(s,1/p), (t,1/q) \in \m H_\eps$ and
\begin{align} \label{eq:embed}
\W_{t+\frac{1}{q}}^q(\O) &= \W_1^{\frac{12}{4 - \eps}}(\O) \subset C^{0,\frac{\eps}{4}}(\O) \cap \W_1^3(\O) \subset C(\oO) \cap \W_1^3(\O)
\end{align}
and
\begin{align} \label{eq:dual_embed}
\W_{2 - s - \frac{1}{p}}^{p'}(\O; D) &= \W_{\eps + \frac{1}{2}}^{\frac{2}{1+\eps}}(\O; D) \subset L^\frac{6}{2+\eps}(\O) \subset L^3(\O)
\end{align}
due to the Sobolev embedding theorem. Note that $\W_{2 - s - \frac{1}{p}}^{p'}(\O)|_\pO = B^{p',p'}_{\eps/2}(\pO)$ has the dual space 
\[
B^{p,p}_{-\eps/2}(\pO) = B^{p,p}_{s-1}(\pO),
\]
cf.~\cite{MM07}. Also
\[
\W_{s+\frac{1}{p}}^p(\O) \subset \W_1^{\frac{6}{2-\eps}}(\O) \subset \W_1^{\frac{12}{4 - \eps}}(\O) = \W_{t+\frac{1}{q}}^q(\O).
\]
Indeed the embedding of $\W_{s+\frac{1}{p}}^p(\O)$ into $\W_{t+\frac{1}{q}}^q(\O)$ is compact for $\eps \in (0,1/2)$ since
\begin{align} \label{eqn:cmpt}
\left( s+\frac{1}{p} \right) - \left( t+\frac{1}{q} \right) = \frac{1}{2} - \eps.
\end{align}
We set
\begin{align*}
Y = \W_{t+\frac{1}{q}}^q(\O;D_\phi) \times \W_{t+\frac{1}{q}}^q(\O;D_u),\qquad
Z = \W_{s+\frac{1}{p}}^p(\O;D_\phi) \times \W_{s+\frac{1}{p}}^p(\O;D_u).
\end{align*}

\begin{corollary}
Let $s$, $p$ and $\kappa \in L^\infty(R)$ be as above. There exists an $\eps = \eps(\pO, D, R)$ in $(0,1/2)$ such that Poisson's equation is well-posed in the spaces
\begin{align} \label{eq:rlin} \left. \begin{array}{rlrl}
v & \in \W_{s+\frac{1}{p}}^p(\O), & \qquad \Delta v & \in \left( \W_{2 - s - \frac{1}{p}}^{p'}(\O; D) \right)^*,\\[3mm]
v|_D & \in B^{p,p}_s(D), & \qquad \kappa v + \p_\nu v|_R & \in B^{p,p}_{s-1}(R)
\end{array} \quad \right\} \end{align}
whenever $(s,1/p) \in \m H_\eps$.
\end{corollary}

\begin{proof}
A standard energy argument ensures that, given Dirichlet boundary conditions on $D$ with data in $B^{p,p}_s(D)$ and Robin boundary conditions with data in $B^{p,p}_{s-1}(R)$, the mixed Poisson problem has a unique solution $v$ in $H^1(\O)$. This function $v$ is also the unique solution of the system with Neumann boundary conditions
\[
h - \kappa v \in L^2(R) \subset \bigl( B^{\frac{2}{1 + \eps},\frac{2}{1 + \eps}}_{\frac{\eps}{2},0}(R) \bigr)^* = \bigl( B^{p',p'}_{1-s,0}(R) \bigr)^* = B^{p,p}_{s-1}(R) \qquad \mbox{on } R,
\]
giving the required regularity by Lemma \ref{thm:mitrea}. The result now follows since $\| \kappa v \|_{B^{p,p}_{s-1}(R)} \lesssim \| v \|_{H^1(\O; D)}$. 
\end{proof}

Now, motivated by \eqref{eq:sphi}, we consider the following modified weak formulation for $(\zp , \zu) \in Y$
\begin{align} \label{eq:wsm} \left. \begin{array}{ccccll}
\la \nabla \phi, \nabla \psi \ra &&& = & \la \frac{\sigma'(u)}{\sigma(u)} \nabla u \cdot \nabla \phi, \psi \ra & \qquad \forall \, \psi,\\[2mm]
\la \nabla u, \nabla w \ra & + & \la \kappa \, u, w \ra_{R_u} & = & \la \sigma(u) \nabla \phi \cdot \nabla \phi, w \ra + \la \hu, w \ra_{R_u} & \qquad \forall \, w,
\end{array} \right\} \end{align}
with $(\psi, w) \in Z$. Notice that again $\zp = \phi - g_\phi$ and $\zu = u - g_u$.

\begin{lemma}
The set of solutions of \eqref{eq:ws} which belong to $Y$ is equal to the set of solutions of \eqref{eq:wsm}. Moreover, if $(\zp, \zu) \in X$ solves \eqref{eq:ws} then it solves \eqref{eq:wsm} for all $(\psi, w) \in \W_1^\infty(\O;D_\phi) \times \W_1^\infty(\O;D_u)$.
\end{lemma}

\begin{proof}
One only needs to consider the first equation of \eqref{eq:ws} and \eqref{eq:wsm}. Let $(\zp , \zu) \in Y$ solve \eqref{eq:wsm}. Let $\{ \vartheta_\eps \}_\eps$ be an approximate identity and $\sigma_\eps := \sigma(u) * \vartheta_\eps$ and $\psi_\eps := \psi \, \sigma_\eps$. Then
\[
\int_\O \nabla \phi \cdot \nabla \psi \, \d x \stackrel{\eqref{eq:wsm}}{=} \lim_{\eps \to 0} \int_\O \Bigl( \frac{\nabla \sigma(u)}{\sigma_\eps} \cdot \nabla \phi \Bigr) \psi \, \d x,
\]
as $\sigma_\eps \to \sigma(u)$ in $L^\infty(\O)$. Now choosing $\psi_\eps$ as test function
\[
\lim_{\eps \to 0} \int_\O \nabla \phi \cdot \nabla \psi_\eps \, \d x = \lim_{\eps \to 0} \int_\O \Bigl( \frac{\nabla \sigma(u)}{\sigma_\eps} \cdot \nabla \phi \Bigr) \psi_\eps \, \d x = \int_\O \bigl( \nabla \sigma(u) \cdot \nabla \phi \bigr) \psi \, \d x,
\]
but also
\begin{align*}
\lim_{\eps \to 0} \int_\O \nabla \phi \cdot \nabla \psi_\eps \, \d x &= \lim_{\eps \to 0} \int_\O \sigma_\eps \nabla \phi \cdot \nabla \psi \, \d x + \int_\O \bigl( \nabla \phi \cdot \nabla \sigma_\eps \bigr) \psi \, \d x \\
&= \int_\O \sigma(u) \nabla \phi \cdot \nabla \psi \, \d x + \int_\O \bigl( \nabla \sigma(u) \cdot \nabla \phi \bigr) \psi \, \d x.
\end{align*}
Subtraction shows that $(\tilde{\phi} , \tilde{u}) \in Y$ solves \eqref{eq:ws}. The other direction follows from re-arranging the above identities; the test spaces $Z$ and $\W_1^\infty(\O;D_\phi) \times \W_1^\infty(\O;D_u)$ have to be adapted to the size of trial function spaces. 
\end{proof}

It is convenient to define the operators
\[ \begin{array}{rcccll}
S_1 : \; & Y & \to &  W_\phi^*& , \  (\zvp, \zv) & \mapsto \bigl( \psi \mapsto \la \textstyle \frac{\sigma'(v)}{\sigma(v)} \nabla v \cdot \nabla \varphi, \psi \ra \bigr),\\[2mm]
S_2 : \; & Z_\phi & \to & W_\phi^* & , \; \phantom{(} \zvp \phantom{, v)} & \mapsto \bigl( \psi \mapsto \la \nabla \varphi, \nabla \psi \ra \bigr),\\[2mm]
S_3 : \; & Y & \to & W_u^* & , \; (\zvp, \zv) & \mapsto \bigl( w \mapsto \la \sigma(v) \nabla \varphi \cdot \nabla \varphi, w \ra + \la \hu, w \ra_{R_u} \bigr),\\[2mm]
S_4 : \; & Z_u & \to & W_u^* & , \; \phantom{(\varphi, } \zv \phantom{)} &  \mapsto \bigl( w \mapsto \la \nabla v, \nabla w \ra + \la \kappa \, v, w \ra_{R_u} \bigr),
\end{array} \]
where
\begin{align*}
Z_\phi & = \W_{s+\frac{1}{p}}^p(\O;D_\phi), & Z_u & = \W_{s+\frac{1}{p}}^p(\O;D_u),\\
W_\phi & = \W_{2 - s - \frac{1}{p}}^{p'}(\O; D_\phi), & W_u & = \W_{2 - s - \frac{1}{p}}^{p'}(\O; D_u).
\end{align*}
The notation indicates that operators map into dual spaces with the associated test functions $\psi$ and $v$. Let $I$ be the identity map and
\begin{align} \label{eq:S}
S := (I, S_4^{-1} \circ S_3) \circ ( S_2^{-1} \circ S_1, I).
\end{align}
Given an initial pair $(\varphi, v)$, $S_1$ returns in the first component the right-hand side of \eqref{eq:wsm} and thus $S_2^{-1} \circ S_1$ gives an update of the first component. This, together with an unchanged $v$, is passed into $S_3$ and then $S_4^{-1}$ to return first an updated right-hand side and then an updated second component.

\begin{lemma} \label{thm:cont}
The operators $S$ maps continuously into $Z$.
\end{lemma}

\begin{proof}
Lemma \ref{thm:mitrea} ensures that $S_2^{-1}$ and $S_4^{-1}$ map continuously into $\W_{s+\frac{1}{p}}^p(\O)$. Notice that, due to \eqref{eq:dual_embed},
\[
\frac{\nabla \sigma(v)}{\sigma(v)} \cdot \nabla \varphi, \;\;\; \sigma(v) |\nabla \varphi |^2 \;\;\; \in L^\frac{6}{4 - \eps}(\O) \;\; = \;\; \left( L^\frac{6}{2+\eps}(\O) \right)^* \subset \left( \W_{2 - s - \frac{1}{p}}^{p'}(\O) \right)^*.
\]
As $h \in B_{s-1}^{p,p}(R_u)$, $\kappa \in L^\infty(R_u)$, $v \in C(\oO)$ also $\kappa v, h \in \bigl( \W_{2 - s - \frac{1}{p}}^{p'}(\O) \bigr)^*$. Therefore $S_1$ and $S_3$ map into $W^*_\phi$ and $W^*_u$. 
\end{proof}

Lemma \ref{thm:cont} gives access to Schauder's fixed point argument provided $\frac{\nabla \sigma}{\sigma}$ is not too large in relation to other parameters of the problem. Let $C_1$ to $C_4$ be the embedding constants of
\[ \begin{array}{rlcrl}
\W_{2 - s - \frac{1}{p}}^{p'}(\O; D_u) & \hookrightarrow L^\frac{12}{8+\eps}(\O), & \qquad & B_s^{p,p}(D_\phi) & \hookrightarrow \W_{2 - s - \frac{1}{p}}^{p'}(\O)^*,\\
\W_{s+\frac{1}{p}}^p(\O) & \hookrightarrow \W_{t+\frac{1}{q}}^q(\O), &\qquad & B_s^{p,p}(D_u) & \hookrightarrow \W_{2 - s - \frac{1}{p}}^{p'}(\O)^*,
\end{array} \]
respectively. By abuse of notation we denote by $\| S_2^{-1} \|$ and $\| S_4^{-1} \|$ the operator norms of the linear parts of $S_2^{-1}$ and $S_4^{-1}$; that is of $S_2^{-1}$ and $S_4^{-1}$ if $g_\phi$ and $g_u$ were $0$.

\begin{lemma} \label{thm:com}
There exists a positive constant
\begin{align} \label{eq:C}
C_* = C_*\bigl(\{C_i\}_{i=1}^4, \| S_2^{-1} \|, \| S_4^{-1} \|, \sigma^\circ, \| g_\phi \|_{B_s^{p,p}}, \| g_u \|_{B_s^{p,p}}, \| h \|_{B_{s-1}^{p,p}} \bigr) \quad
\end{align}
such that whenever $\| \frac{\sigma'}{\sigma} \|_{L^\infty(\R)} \le C_*$ then there is a ball $B \subset Y$ such that $S$ maps $B$ into $B$.
\end{lemma}

\begin{proof}
By H{\"o}lder's inequality
\begin{align*}
\bigl\| \textstyle \frac{\nabla \sigma(v)}{\sigma(v)} &\cdot \nabla \varphi \bigr\|_{\W_{2 - s - \frac{1}{p}}^{p'}(\O; D_\phi)^*} =  \sup_{\psi \neq 0} \frac{\la \textstyle \frac{\sigma'(v)}{\sigma(v)} \nabla v \cdot \nabla \varphi, \psi \ra}{\| \psi \|_{\W_{2 - s - \frac{1}{p}}^{p'}(\O; D_\phi)}}
\\
&\le C_1 \, \| \textstyle \frac{\sigma'}{\sigma} \|_{L^\infty(\R)} \; \| \nabla v \|_{L^\frac{12}{4-\eps}(\O)} \; \| \nabla \varphi \|_{L^\frac{12}{4-\eps}(\O)}\\
&\lesssim  \tst \, C_1 \, \| \frac{\sigma'}{\sigma} \|_{L^\infty(\R)} \, \bigl( \| (\zvp, \zv) \|_Y^2 + C_2 \| g_\phi \|_{B_s^{p,p}(D_\phi)}^2 + C_4 \| g_u \|_{B_s^{p,p}(D_u)}^2 \bigr).
\end{align*}
Similarly,
\begin{align} \label{eq:phi_dep}
\bigl\| \textstyle \sigma(v) \nabla \varphi \cdot \nabla \varphi \bigr\|_{\W_{2 - s - \frac{1}{p}}^{p'}(\O; D_u)^*} = & \, \sup_{w \neq 0} \frac{\la \textstyle \sigma(v) \nabla \varphi \cdot \nabla \varphi, w \ra}{\| w \|_{\W_{2 - s - \frac{1}{p}}^{p'}(\O; D_u)}}
\\ \nonumber
\lesssim & \, C_1 \, \sigma^\circ \, \bigl( \| \nabla \zvp \|_{L^\frac{12}{4-\eps}(\O)}^2 + C_2 \| g_\phi \|_{B_s^{p,p}(D_\phi)}^2 \bigr).
\end{align}
We need to bound the first component $I \circ S_2^{-1} \circ S_1 = S_2^{-1} \circ S_1$ and the second component $S_4^{-1} \circ S_3 \circ ( S_2^{-1} \circ S_1, I)$ of $S$, cf.~\eqref{eq:S}. Suppose that $(\zvp, \zv)$ are contained in the ball $B = \{ y \in Y : \| y \|_Y \le r \}$. As $S_2$ is an invertible affine function, there are generic constants $C$ with a parameter dependence as indicated in \eqref{eq:C} such that
\begin{align} \label{eq:phi_bound}
\| ( S_2^{-1} \circ S_1 ) (\zvp, \zv) \|_{\W_{t+\frac{1}{q}}^q(\O;D_\phi)} \le \tst C \, \| \frac{\sigma'}{\sigma} \|_{L^\infty(\R)} \, r^2 + C.
\end{align}
For the second component notice that the right-hand side of \eqref{eq:phi_dep} only depends on $\varphi$ and not $v$
\[
\| ( S_4^{-1} \circ S_3 \circ ( S_2^{-1} \circ S_1, I)) (\zvp, \zv) \|_{\W_{t+\frac{1}{q}}^q(\O;D_u)} \le \tst C \bigl( \underbrace{\tst C \, \| \frac{\sigma'}{\sigma} \|_{L^\infty(\R)} \, r^2 + C}_{\mbox{\scriptsize owing to \eqref{eq:phi_bound}}} \bigr)^2 + C.
\]
Therefore $\| ( S_2^{-1} \circ S_1 ) (\zvp, \zv) \|$ is bounded by a quartic polynomial of the form
\[
\tst C \, \| \frac{\sigma'}{\sigma} \|_{L^\infty(\R)}^2 \, r^4 +C \, \| \frac{\sigma'}{\sigma} \|_{L^\infty(\R)} \, r^2 + C.
\]
At radii where it intersects the first diagonal $r \mapsto r$ the operator $S$ maps $B$ into $B$. The existence of such an intersection point is guaranteed if $\| \frac{\sigma'}{\sigma} \|_{L^\infty(\R)}$ is sufficiently small. 
\end{proof}

The above lemma is consistent with the analysis of the linear problem with a constant $\sigma$ in the sense that a suitable $C_*$ can always be found as $\frac{\sigma'}{\sigma}$ becomes small. We now turn to Schauder's fixed point theorem to obtain the existence of solutions in $Z$.

\begin{theorem} \label{thm:reg}
There exists a positive constant
\[
C_* = C_*\bigl(\{C_i\}_{i=1}^4,\| S_2^{-1} \|, \| S_4^{-1} \|, \sigma^\circ, \| g_\phi \|_{B_s^{p,p}(D_\phi)}, \| g_u \|_{B_s^{p,p}(D_u)}, \| h \|_{B_{s-1}^{p,p}(R_u)}  \bigr)
\]
such that whenever $\| \frac{\sigma'}{\sigma} \|_{L^\infty(\R)} \le C_*$ then \eqref{eq:wsm} has a solution in $Z$.
\end{theorem}

\begin{proof}
The compactness of $S$ follows from Lemma \ref{thm:cont} and \eqref{eqn:cmpt}. Now the result is a consequence of Lemma \ref{thm:com} and Schauder's fixed point theorem. 
\end{proof}

While the global regularity estimate in $Y$ is sharp in the setting of creased domains (up to the distance of $(t,1/q)$ to the boundary of $\m H_\eps$), as a comparison with the Poisson problem shows, more regularity is seen away from the boundary. For the next theorem we assume $(\tilde{\phi}, \tilde{u}) \in Y$; however, it is not relevant whether this is established with the above fixed-point argument or otherwise.

\begin{theorem} \label{thm:loc}
Let $\Omega_0$ be a relatively compact Lipschitz domain in $\O$: $\O_0 \Subset \O$. Let $(\zp, \zu) \in Y$ be a solution of \eqref{eq:wsm}. Then $\zp, \zu \in W_2^s(\O_0)$ for all $s \in (1,\infty)$.  If $\sigma \in C_0^\infty(\R)$ then $\zp, \zu \in C^\infty(\overline{\O_0})$.
\end{theorem}

\begin{proof}
Let $\{ \O_i \}_{i \in \N}$ and $\O_\infty$ be smooth domains with $\O_i \Subset \O_{i+1}$ and $\O_i \subset \O_\infty \Subset \O$ for all $i \in \N$. Without loss of generality we may assume that the boundary data $g_\phi$ and $g_u$ have extensions from the boundary onto $\O$ such that $g_\phi, g_u \in C^\infty(\O_\infty)$. Fix $i \in \N \setminus \{ 0 \}$. Let $\zeta_i$ be a smooth function $\O_i \to [0,1]$ such that $\zeta_i|_{\pO_i} = 0$ and $\zeta_i|_{\O_{i-1}} = 1$. Then $\zeta_i \zp$ and $\zeta_i \zu$ solve Poisson's problem on $\O_i$ with homogeneous Dirichlet boundary conditions and the right-hand sides
\[ \begin{array}{l}
\zeta_i \bigl( \frac{\sigma'(\zu)}{\sigma(\zu)} \nabla \zu \cdot \nabla \zp \bigr) - \zeta_i \Delta g_\phi - 2 \nabla \zeta_i \cdot \nabla \zp - \zp \Delta \zeta_i,\\[2mm]
\zeta_i \bigl( \sigma(\zu) \nabla \zp \cdot \nabla \zp \bigr) - \zeta_i \Delta g_u - 2 \nabla \zeta_i \cdot \nabla v - v \Delta \zeta_i
\end{array} \]
in $L^{\frac{6}{4 - \eps}}(\O_i)$, respectively. According to Theorem 9.15 in \cite{GT01} a solution to Poisson's problem with homogeneous Dirichlet conditions belongs to $\W_2^{\frac{6}{4 - \eps}}(\O_i)$ if the right-hand side belongs to $L^{\frac{6}{4 - \eps}}(\O_i)$ and $\frac{6}{4 - \eps} \in (1,\infty)$. Therefore $\zp, \zu \in \W_2^{\frac{6}{4 - \eps}}(\O_i) \subset \W_1^{\frac{12}{4 - 2 \eps}}(\O_i)$. Substituting $\eps$ by $2 \eps$ one may pass from $i$ to $i-1$ and repeat the argument. We conclude via induction that $\zp, \zu \in \W_2^s(\O_1)$ for any $s \in (1,\infty)$, recalling that a negative $\frac{12}{4 - 2 \eps}$ corresponds to a right-hand side in $L^\infty(\O_i)$. Indeed $\zp, \zu \in \W_2^s(\O_i)$ for all $i \in \N$ and $s \in (1,\infty)$.

Now let $\sigma \in C^\infty(\R)$. Theorem 9.19 in \cite{GT01} states that if the right-hand side in Poisson's problem is in $\W_k^s(\O_i)$ then the solution belongs to $\W_{k+2}^s(\O_i)$ with $k \in \{1, 2, 3, \ldots \}$ and $s \in (1, \infty)$. Leibniz' rule $(f \cdot g)^{(k)} = \sum_{j = 0}^k \left( k \atop j \right) f^{(j)} \cdot g^{(k-j)}$ shows that for a given $s \in (1,\infty)$ a $s' \in (1,\infty)$ can be chosen such that $\frac{\nabla \sigma(\zu)}{\sigma(\zu)} \nabla \zp, \sigma(\zu) \nabla \zp \cdot \nabla \zp \in \W_k^s(\O_i)$ if $\zp, \zu \in \W_{k+1}^{s'}(\O_i)$. Hence induction over $k$, coupled with a shift from $\O_i$ to $\O_{i-1}$ as above to impose smooth boundary conditions, shows that $\zp, \zu \in \W_k^s(\O_0)$ for all $k \in \N$ and $s \in (1,\infty)$. Use of the Sobolev embedding theorem concludes the proof. 
\end{proof}

\begin{remark}
Assumption (C) is made to establish sufficient regularity of elliptic equations with Lemma \ref{thm:mitrea} in the context of non-smooth domains and mixed boundary conditions. Also in other settings corresponding elliptic regularity results are available and the above analysis can be transferred with minor modifications. We point for the pure Dirichlet problem to \cite{JK95,FMM98}, for the Neumann problem to \cite{JK82,FMM98}. A related approach for the mixed problem on smooth, non-creased domains is proposed in~\cite{S97}.
\end{remark}

% ------------------------------------------------------------------------------------------------------------------------
\section{{\em A Priori} and {\em A Posteriori} Error Analysis}
% ------------------------------------------------------------------------------------------------------------------------

In this section we present {\em a priori} and {\em a posteriori} error bounds for conforming finite element approximations. We first present in an abstract form that the Galerkin method is quasi-optimal and that the error $x-x_n$ can be bounded using the dual norm of the residual of the approximation. In the second part we choose a particular approximation technique, namely a conforming $h$-adaptive finite element method. We use interpolation estimates to bound the errors in terms of mesh size and polynomial degree.

\subsection{Abstract error bounds}

In the theorem we assume that $(\zp,\zu)\in Y$. There is also an assumption on small data which 
relaxes as the coupling of the equations, measured by the Lipschitz constant of $\sigma$ denoted $C_7$, weakens. In order to get the correct dependency of $C_7$ we introduce a scaling factor ($\tau$) of the second component of equation (\ref{eq:wso}) in the proof of the theorem. Beside the Lipschitz constant $C_7$ we let $C_{8}$ be the embedding constant from $H^1$ into $L^6$, $C_{9}$ be a Poincar\'{e}-Friedrichs constant and
\begin{align}\label{eq:consts}
C_5(\phi)&=C_7 C_{8}(1+C_{9})\|\nabla \phi\|_{L^3(\Omega)}\max(1,g^\circ-g_\circ+\|\nabla g_\phi\|_{L^3(\Omega)}),\\ \nonumber
C_6(\phi)&=\sigma^\circ((1+C_{9})\big(C_{8}\|\nabla\phi\|_{L^3(\Omega)}+C_{8}\|\nabla g_\phi\|_{L^3(\Omega)})+g^\circ-g_\circ\big),
 \end{align}
We note that $C_5$ is directly proportional to $C_7$.
\begin{theorem}\label{thm:apriori}
Suppose there is a solution $x=(\tilde{\phi},\tilde{u})\in Y$ of equation (\ref{eq:wsg}) which satisfies
\[
C_5(\phi)\leq\frac{(1-\delta)^2\sigma_0}{C_6(\phi)+(1-\delta)\sigma_0}
\]
for a $\delta \in (0,1)$ with the constants from equation (\ref{eq:consts}). Then the solution is unique. Furthermore, if $x_n\in X_n$ is its Galerkin approximation then the following {\em a priori} and {\em a posteriori} error bounds hold:
\begin{align*}
\| x-x_n\|_X \lesssim & \, \inf_{y_n\in X_n}\|x-y_n\|_X,\\
\| x-x_n\|_X \lesssim & \, \|L x_n+N x_n-b\|_{X^*}.
\end{align*}
\end{theorem}
\begin{proof}
We introduce for $(\tilde{\varphi}, \tilde{v}) \in X$ the norm:
\[
\|(\tilde{\varphi}, \tilde{v})\|^2_{X^\tau}:= \sigma_0\|\nabla \tilde{\varphi} \|^2_{L^2(\Omega)}+\tau^2 \| \nabla \tilde{v} \|^2_{L^2(\Omega)} + \tau^2 \|\sqrt{\kappa} \tilde{v} \|^2_{L^2(R_u)}.
\]
We pick a solution $x=(\tilde{\phi},\tilde{u})\in Y$. For any function pair $(\psi_n,v_n)=(\tilde{\psi}_n+g_\phi,\tilde{v}_n+g_u)$ with  $(\tilde{\psi}_n,\tilde{v}_n)\in X_n$, we have,
\begin{align*}
&\|(\phi-\phi_n,u-u_n)\|^2_{X^\tau}\leq \la \sigma(u_n)\nabla(\phi-\phi_n),\nabla(\phi-\phi_n)\ra\\
&\qquad +\tau^2\|\nabla(u-u_n)\|^2_{L^2(\Omega)}+\tau^2\|\sqrt{\kappa} (u-u_n)\|^2_{L^2(R_u)} \\
&\leq \la\sigma(u_n)\nabla(\phi-\phi_n),\nabla(\phi-\psi_n)\ra+\tau^2\la\nabla(u-u_n),\nabla(u-v_n)\ra
\\
&\qquad +\tau^2\la\kappa (u-u_n),u-v_n\ra_{R_u} +\la\sigma(u_n)\nabla(\phi-\phi_n),\nabla(\psi_n-\phi_n)\ra\\
&\qquad +\tau^2\la\nabla(u-u_n),\nabla(v_n-u_n)\ra+\tau^2\la\kappa ( u-u_n),v_n-u_n\ra_{R_u}\\
&\stackrel{(*)}{\leq} \frac{\sigma^\circ}{\sigma_\circ}\|(\phi-\phi_n,u-u_n)\|_{X^\tau}\|(\phi-\psi_n,u-v_n)\|_{X^\tau}
\\
&\qquad + C_7 C_{8}\|\nabla \phi\|_{L^3(\Omega)}\|u-u_n\|_{H^1(\Omega)}\|\nabla(\phi_n-\psi_n)\|_{L^2(\Omega)}\\
&\qquad +\tau^2\|\sigma(u)\lceil\tilde{\phi}\rceil\nabla \phi-\sigma(u_n)\lceil\tilde{\phi}_n\rceil\nabla \phi_n\|_{L^2(\Omega)}\|\nabla (v_n-u_n)\|_{L^2(\Omega)}\\
&\qquad +C_{8} \tau^2\|\nabla g_\phi\|_{L^3(\Omega)}\|\sigma(u)\nabla\phi-\sigma(u_n)\nabla\phi_n\|_{L^2(\Omega)}\|v_n-u_n\|_{H^1(\Omega)}\\
&\leq \frac{\sigma^\circ}{\sigma_\circ}\|(\phi-\phi_n,u-u_n)\|_{X^\tau}\|(\phi-\psi_n,u-v_n)\|_{X^\tau}
\\
&\qquad +C_6\tau^2\|\nabla (\phi-\phi_n)\|_{L^2(\Omega)}\|\nabla (v_n-u_n)\|_{L^2(\Omega)}\\
&\qquad +C_5\big(\|\nabla(\phi_n-\psi_n)\|_{L^2(\Omega)}+\tau^2\|\nabla (v_n-u_n)\|_{L^2(\Omega)}\big)\|\nabla(u-u_n)\|_{L^2(\Omega)},
\end{align*}
where we in $(*)$ use that $\la \sigma(u)\nabla\phi,\nabla \tilde{\psi}\ra=\la \sigma(u_n)\nabla\phi_n,\nabla \tilde{\psi}\ra=0$ for any $\tilde{\psi}\in X_n$.
We now use the triangle inequality on $\|\nabla(\psi_n-\phi_n)\|_{L^2(\Omega)}\leq \|\nabla(\phi-\phi_n)\|_{L^2(\Omega)}+\|\nabla(\phi-\psi_n)\|_{L^2(\Omega)}$ and $\|\nabla(v_n-u_n)\|_{L^2(\Omega)}\leq \|\nabla(u-u_n)\|_{L^2(\Omega)}+\|\nabla(u-v_n)\|_{L^2(\Omega)}$ and introduce 
$C_{10}=\sigma^\circ\sigma_\circ^{-1}+C_6\sigma_\circ^{-1}\tau+C_5\sigma_\circ^{-1}\tau^{-1}+C_5\tau^{-1}$ to get,
\begin{align*}
&\|(\phi-\phi_n,u-u_n)\|^2_{X^\tau}\leq C_{10}\|(\phi-\phi_n,u-u_n)\|_{X^\tau}\|(\phi-\psi_n,u-v_n)\|_{X^\tau}
\\
&\qquad +C_6\tau^2\|\nabla (\phi- \phi_n)\|_{L^2(\Omega)}\|\nabla (u-u_n)\|_{L^2(\Omega)}\\
&\qquad +C_5\big(\|\nabla(\phi-\phi_n)\|_{L^2(\Omega)}+\tau^2\|\nabla (u-u_n)\|_{L^2(\Omega)}\big)\|\nabla(u-u_n)\|_{L^2(\Omega)},\\
&\leq
C_{10}\|(\phi-\phi_n,u-u_n)\|_{X^\tau}\|(\phi-\psi_n,u-v_n)\|_{X^\tau} \\&\qquad +\frac{1}{2\epsilon}(C_5+C_6 \tau^2)\|\nabla(\phi-\phi_n)\|^2_{L^2(\Omega)}\\
&\qquad +\big(\frac{\epsilon}{2}(C_5+C_6 \tau^2)+C_5 \tau^2\big)\|\nabla(u-u_n)\|^2_{L^2(\Omega)}.
\end{align*}
for any $\epsilon>0$. We now want to find the maximum value of $C_5$ while fulfilling $\frac{1}{2\epsilon}(C_5+C_6 \tau^2) \leq (1-\delta)\sigma_0$ and $\big(\frac{\epsilon}{2}(C_5+C_6 \tau^2)+C_5 \tau^2\big) \leq (1-\delta)\tau^2$, for some $0<\delta<1$, since  we then can subtract $(1-\delta)\|(\phi-\phi_n,u-u_n)\|^2_{X^\tau}$ on both sides of the equality sign. Algebraic manipulation reveals that by choosing $\epsilon$ equal to $\epsilon^*=\frac{1-\delta}{C_6+(1-\delta)\sigma_0}>0$ and  $\tau^2=(\tau^*)^2:=\frac{\epsilon^*(1-\delta)\sigma_0}{C_6}>0$ leads to a value $C_5=\frac{(1-\delta)^2\sigma_0}{C_6+(1-\delta)\sigma_0}$. Under the assumptions in the statement of the theorem we therefore have,
\[
\|(\phi-\phi_n,u-u_n)\|_{X^{\tau^*}}\leq \delta^{-1}C_{10}(\tau^*,\sigma_0,\sigma^0,C_5,C_6)\|(\phi-\psi_n,u-v_n)\|_{X^{\tau^*}}
\]
for that fixed $\tau^*>0$.
The {\em a priori} bound now follows using the Poincar\'{e}-Friedrichs inequality and simple algebraic manipulation, since the $X^\tau$-norm and the $X$-norm are equivalent for fixed $\tau$.

Suppose we have two solutions $x_1$ and $x_2$. The {\em a priori} bound gives that any Galerkin approximation will eventually get arbitrary close to both solutions which means that they coincide, i.e.~$x_1=x_2$.

Only the {\em a posteriori} bound now remains to prove. It follows continuing ideas as in the {\em a priori} bound. We have, following the step $(*)$ and below,
\begin{align*}
\|(&\phi-\phi_n,u-u_n)\|^2_{X^\tau}\leq \la \sigma(u_n)\nabla(\phi-\phi_n),\nabla(\phi-\phi_n)\ra\\
&\qquad +\tau^2\la\nabla(u-u_n),\nabla(u-u_n)\ra+\tau^2\la\kappa (u-u_n),(u-u_n)\ra_{R_u}\\
&\leq
\la b-L(\tilde{\phi}_n,\tilde{u}_n)-N(\tilde{\phi}_n,\tilde{u}_n),(\phi-\phi_n,\tau^2(u-u_n))\ra
\\
&\qquad +C_6 \tau^2 \|\nabla(\phi-\phi_n)\|_{L^2(\Omega)}\|\nabla(u-u_n)\|_{L^2(\Omega)}
\\ &\qquad
+C_5\|\nabla(u-u_n)\|_{L^2(\Omega)}\left(\|\nabla(\phi-\phi_n)\|_{L^2(\Omega)}+\tau^2\|\nabla(u-u_n)\|_{L^2(\Omega)}\right),
\end{align*}
Under the assumptions of the theorem we can now repeat the exact same argument as above picking $\epsilon^*$ and $\tau^*$ in the same way. We get,
\begin{align*}
\|x-x_n\|^2_{X}&\lesssim
\|x-x_n\|^2_{X^{\tau^*}}
\leq \delta^{-1}\max(1,(\tau^*)^2)\la b-L x_n-N x_n,x-x_n\ra\\
&\lesssim \|L x_n+N x_n-b\|_{X^*}\|x-x_n\|_X,
\end{align*} 
completing the proof of the theorem. 
\end{proof}

\begin{remark}
Since $C_5(\phi)$ and $C_6(\phi)$ depend on $\|\nabla \phi\|_{L^3(\Omega)}$ a bound on this norm needs to be computed in order to obtain a computable {\em a posteriori} error bound. In the setting of creased domains this can in principle be derived from the radius $r$ in Lemma \ref{thm:com}.
\end{remark}

\subsection{Error bounds for $\s h$-adaptive finite element approximations}

The local regularity result presented in Theorem \ref{thm:loc} indicates that the problem is well suited for adaptive finite element methods. Suppose that $\mathcal{T}_\phi$ and $\mathcal{T}_u$ are decompositions of $\Omega$ into tetrahedrons. We let $\s{h}_T$ denote the diameter of element $T$ and assume that $\mathcal{T}_\phi$ and $\mathcal{T}_u$ are nondegenerate, i.e. there is a mesh-independent constant $\gamma$ such that
\[\max_{T}\frac{\s{h}_T}{d_T}\leq\gamma,\]
where $d_T$ is the diameter of the largest ball contained in $T$. We denote the set of elements $T'$  neighbouring $T$ by $S_T=\{\cup T: T' \cap T\neq\emptyset\}$, assuming that elements are closed. Further let $\mathcal{E}_{I_\phi}$, $\mathcal{E}_{I_u}$  be the set of all interior facets of the two meshes and let $\mathcal{E}_{N_\phi}$, $\mathcal{E}_{R_u}$ be the set of boundary facets on $N_\phi$ and $R_u$, respectively. It is implicitly supposed here that $N_\phi$ and $R_u$ are unions of edges. We denote the set of all elements neighboring $e$ by $S_e=\{\cup e:T\cap e\neq\emptyset\}$. The diameter of a facet $e$ is denoted $\s{h}_e$. The spaces of polynomials of total or partial degree less than or equal to $k$, defined on $T$, is denoted $P_k(T)$. Let
\begin{align}\label{Xn} \left. \begin{array}{rl}
P_n := & \{ \varphi \in C(\Omega): \varphi |_T \in P_k(T) \mbox{ and } \varphi|_{D_\phi} = 0\},\\[1mm]
U_n := & \{v \in C(\Omega): v|_T \in P_\ell(T) \mbox{ and } v|_{D_u} = 0\},
\end{array} \quad \right\} \end{align}
and $X_n = P_n \times U_n$. We let $\pi_\phi: H^1(\Omega)\rightarrow P_n$ and $\pi_u: H^1(\Omega)\rightarrow U_n$ be Scott-Zhang interpolants, as defined in \cite{ScZh90}. We recall the interpolation bounds
\begin{align}\label{int1} \left. \begin{array}{lll}
\|v-\pi_u v\|_{W_2^s(T)} & \, \lesssim \s{h}_T^{r-s} & |v|_{W^{r}_2(S_T)},\\[1mm]
\|v-\pi_u v\|_{L^2(e)} & \, \lesssim \s{h}_e^{1/2} & |v|_{W^1_2(S_e)},
\end{array} \quad \right\} \end{align}
for $v\in W_2^r(\Omega; D_u)$, $r \in [1, k+1]$, $s \in [0,r]$ and where $|\cdot|_{W^{r}_2}$ is the semi-norm only including the derivatives of order $r$. An analogous bound also holds for $\pi_\phi$. From here on we drop the subscripts of $\pi_u$ and $\pi_\phi$ since it will be clear from the context which operator is meant.

\begin{theorem}
Let $x=(\tilde{\phi},\tilde{u})\in Y$ be a solution of (\ref{eq:wsg}) and $x_n\in X_n$ be its Galerkin approximation. Then we have the {\em a priori} error estimate,
\[
\| x-x_n\|_X^2 \lesssim \sum_{T\in\mathcal{T}_\phi}\s{h}_T^{2(r-1)}|\tilde{\phi}|^2_{W_{2}^{r}(S_T)}+\sum_{T\in\mathcal{T}_u}\s{h}_T^{2(s-1)}|\tilde{u}|^2_{W_{2}^{s}(S_T)},
\]
with $1\leq r\leq k+1$ and $1\leq s\leq \ell+1$, and the {\em a posteriori} error estimate,
\begin{align*}
\| x-x_n\|_X^2 \lesssim \, \phantom{+} & \sum_{T\in\mathcal{T}_\phi}\s{h}_T^2\eta_{T}^2+\sum_{e\in\mathcal{E}_{I_\phi}}\s{h}_e\eta_{e,I}^2+\sum_{e\in\mathcal{E}_{N_\phi}}\s{h}_e\eta_{e,N}^2\\
 + & \sum_{T\in\mathcal{T}_u}\s{h}_T^2\rho_{T}^2+\sum_{e\in\mathcal{E}_{I_u}}\s{h}_e\rho_{e,I}^2+\sum_{e\in\mathcal{E}_{R_u}}\s{h}_e\rho_{e,R}^2,
\end{align*}
where
\begin{align*}
\eta_{T}(\phi_n,u_n)=&\,\|\nabla\cdot\sigma(u_n)\nabla\phi_n\phantom{[}\|_{L^2(T)},\\
\eta_{e,I}(\phi_n,u_n)=&\,\|\nu\cdot[\sigma(u_n)\nabla \phi_n]\|_{L^2(e)},\\
\eta_{e,N}(\phi_n,u_n)=&\,\|\nu\cdot\phantom{[}\sigma(u_n)\nabla \phi_n\phantom{]}\|_{L^2(e)},\\
\rho_{T}(\phi_n,u_n)=&\,\|\nabla \cdot (\nabla u_n\,+\,\sigma(u_n)\lceil \tilde{\phi}_n\rceil\nabla\phi_n)+\sigma(u_n)\nabla g_\phi\cdot\nabla\phi_n\|_{L^2(T)},\\
\rho_{e,I}(\phi_n,u_n)=&\,\| \nu \cdot ( [\nabla u_n] + \sigma(u_n) \lceil \tilde{\phi}_n\rceil[\nabla\phi_n] ) \|_{L^2(e)},\\
\rho_{e,R}(\phi_n,u_n)=&\,\| \nu\cdot( \phantom{[}\nabla u_n\phantom{]} + \sigma(u_n)\lceil \tilde{\phi}_n\rceil\phantom{[}\nabla\phi_n\phantom{]})+\kappa u_n-h\|_{L^2(e)}.
\end{align*}
\end{theorem}
\begin{proof}
From Theorem \ref{thm:apriori} we have,
\[
\| x-x_n\|_X\lesssim \inf_{y_n\in X_n}\|x-y_n\|_X\leq \big(\|\tilde{\phi}-\pi\tilde{\phi}\|^2_{H^1}+\|\tilde{u}-\pi \tilde{u}\|^2_{H^1}\big)^{1/2}.
\]
The {\em a priori} part of the theorem follows from (\ref{int1}) with~$s=1$.

We turn to the {\em a posteriori} bound. For any $y=(\tilde{\varphi},\tilde{v})\in X$ we have
\begin{align*}
\langle L x_n+N x_n-b,y\rangle= \phantom{+} & \, \langle \sigma(u_n)\nabla\phi_n,\nabla (\tilde{\varphi}-\pi\tilde{\varphi})\rangle +\langle \nabla u_n,\nabla (\tilde{v}-\pi \tilde{v})\rangle\\
+ & \, \langle\kappa u_n,\tilde{v} -\pi \tilde{v}\rangle +\langle \sigma(u_n)\lceil \tilde{\phi}_n\rceil\nabla\phi_n,\nabla (\tilde{v}-\pi \tilde{v})\rangle\\[1mm]
- & \, \langle\sigma(u_n)\nabla g_\phi\cdot\nabla\phi_n,\tilde{v}-\pi \tilde{v}\rangle-\langle h,\tilde{v}-\pi \tilde{v}\rangle_{R_u}.
\end{align*}
We can subtract the interpolants because of Galerkin orthogonality. We apply Green's formula on the elements of the meshes together with the Cauchy-Schwarz inequality to obtain
\begin{align*}
&\langle L x_n+N x_n-b,y\rangle\lesssim\sum_{T\in\mathcal{T}_\phi}\| \nabla\cdot (\sigma(u_n)\nabla\phi_n) \|_{L^2(T)} \, \|\tilde{\varphi}-\pi\tilde{\varphi}\|_{L^2(T)}\\
&\  +\sum_{e\in\mathcal{E}_{I_\phi}}\| \, \nu\cdot[\sigma(u_n)\nabla \phi_n]\|_{L^2(e)}\|\tilde{\varphi}-\pi\tilde{\varphi}\|_{L^2(e)}\\
&\  +\sum_{e\in\mathcal{E}_{N_\phi}}\| \nu\cdot\phantom{[}\sigma(u_n)\nabla \phi_n\phantom{]}\|_{L^2(e)}\|\tilde{\varphi}-\pi\tilde{\varphi}\|_{L^2(e)}\\
&\ +\sum_{T\in\mathcal{T}_u}\| \nabla \cdot (\nabla u_n\,+\,\sigma(u_n)\lceil \tilde{\phi}_n\rceil\nabla\phi_n))+\sigma(u_n)\nabla g_\phi\cdot\nabla\phi_n\|_{L^2(T)} \,
\|\tilde{v}-\pi \tilde{v}\|_{L^2(T)}\\
&\  +\sum_{e\in\mathcal{E}_{I_u}} \| \nu \cdot ( [\nabla u_n] + \sigma(u_n) \lceil \tilde{\phi}_n\rceil[\nabla\phi_n] )\|_{L^2(e)} \, \|\tilde{v}-\pi \tilde{v}\|_{L^2(e)}\\
&\
+\sum_{e\in\mathcal{E}_{R_u}}\| \nu\cdot(\nabla u_n+ \sigma(u_n)\lceil \tilde{\phi}_n\rceil\nabla\phi_n)+\kappa u_n-h \|_{L^2(e)} \, \|\tilde{v}-\pi \tilde{v}\|_{L^2(e)}
\end{align*}
Use of \eqref{int1}, with $s=0$ and $r=1$, and the Cauchy-Schwarz inequality gives
\begin{align*}
\frac{\langle L x_n+N x_n-b,y\rangle}{\|y\|_X}&\lesssim \bigl(\sum_{T\in\mathcal{T}_\phi}\s{h}_T^2\eta_{T}^2+\sum_{e\in\mathcal{E}_{I_\phi}}\s{h}_e\eta_{e,I}^2+\sum_{e\in\mathcal{E}_{N_\phi}}\s{h}_e\eta_{e,N}^2\bigr)^{1/2}\\
&\qquad +\bigl(\sum_{T\in\mathcal{T}_u}\s{h}_T^2\rho_{T}^2+\sum_{e\in\mathcal{E}_{I,u}}\s{h}_e\rho_{e,I}^2+\sum_{e\in\mathcal{E}_{R_u}}\s{h}_e\rho_{e,R}^2\bigr)^{1/2}.
\end{align*}
The theorem follows by taking supremum over all $y\in X$. 
\end{proof}

\begin{remark}
If the Dirichlet data $g_\phi$ and $g_u$ are not traces of finite element functions then an extra data error term will appear in the error estimates. Similarly if exact quadrature is not used additional error terms have to be considered. We have neglected these standard terms in the analysis to make it more readable. For more details, see \cite{ciarlet}.
\end{remark}

\begin{remark}
For additional flexibility one can consider $hp$ finite element spaces. Theorem \ref{thm:loc} indicates that even though the choice of the space $Y$ is sharp, the interior regularity of the solution is higher. This is an ideal setting for $hp$-finite element methods. The above {\em a priori} and {\em a posteriori} error bounds can be transferred to $hp$ approximation spaces for which suitable interpolation operators are available. Some ideas for the construction of such operators are collected in~\cite{MeSa11} and references therein.
\end{remark}

% ------------------------------------------------------------------------------------------------------------------------

% ------------------------------------------------------------------------------------------------------------------------


\begin{thebibliography}{99}

\bibitem{AkLa05} G.~Akrivis and S.~Larsson,
{\em Linearly implicit finite element methods for the time-dependent {J}oule heating problem},
BIT, 45, 429--442, (2005).

\bibitem{AC94} S.N.~Antontsev and M.~Chipot,
{\em The thermistor problem: existence, smoothness, uniqueness, blowup},
SIAM J.~Math.~Anal., 25(4):1128--1156, (1994).

\bibitem{AY06} W.~Allegretto and N.~Yan,
{\em A posteriori error analysis for FEM of thermistor problems,} Int.~J.~Numer.~Anal.~Model.,
3(4):413--436, (2006).

\bibitem{BoOr08} L.~Boccardo, L.~Orsina, and A.~Poretta,
{\em Existence of finite energy solutions for elliptic systems with $L^1$ valued nonlinearities},
Math.~Models Methods Appl.~Sci., 18(5):669--687, (2008).

\bibitem{B68} F.E.~Browder,
{\em Nonlinear eigenvalue problems and Galerkin approximations},
Bull. Amer. Math. Soc., 74:651--656, (1968).

\bibitem{ciarlet} P.G.~Ciarlet,
{\em The finite element method for elliptic problems},
North-Holland Publishing Co., (1978).

\bibitem{CI88} G.~Cimatti,
{\em A Bound for the temperature in the thermistor problem.}
IMA J. Appl. Math., 40(1):15--22, (1988).

\bibitem{CI89} G.~Cimatti,
{\em Remark on existence and uniqueness for the thermistor problem under mixed boundary conditions},
Quart. Appl. Math., 47(1):117--121, (1989).

\bibitem{CI02} G.~Cimatti,
{\em Stability and multiplicity of solutions for the thermistor problem},
Ann. Mat. Pura Appl. IV, 181(2):181--212, (2002).

\bibitem{DuSch88} N.~Dunford and J.T.~Schwartz,
{\em Linear operators},
vol.~1, Wiley, New York, (1988).

\bibitem{ElLa95} C.M.~Elliott and S.~Larsson,
{\em A finite element model for the time-dependent {J}oule heating problem},
Math.~Comp., 64:1433--1453, (1995).

\bibitem{FMM98} E.~Fabes, O.~Mendez and M.~Mitrea,
{\em Boundary layers on {S}obolev-{B}esov spaces and {P}oisson's equation for the {L}aplacian in Lipschitz domains},
J.~Funct.~Anal., 159(2):323--368, (1998).

\bibitem{GaHe94} T.~Gallou\"{e}t and R.~Herbin,
{\em Existence of a solution to a coupled elliptic system},
Appl. Math. Lett. 7(2):49--55, (1994).

\bibitem{GT01} D.~Gilbarg and N.S.~Trudinger,
{\em Elliptic partial differential equations of second order},
Springer, Berlin, (2001).

\bibitem{HeTi06} V.A.~Henneken, M.~Tichem, and P.M.~Sarro,
{\em In-package {MEMS}-based thermal actuators for micro-assembly},
J.~Micromech.~Microeng., 16:6, 107--115, (2006).

\bibitem{HoLa10} M.J.~Holst, M.G.~Larson, A.~M\aa lqvist, and R.~S\"{o}derlund,
{\em Convergence analysis of finite element approximations of the {J}oule heating problem in three spatial dimensions},
BIT Numer.~Math., 50:781--795, (2010).

\bibitem{HRS93} S.D.~Howison, J.F.~Rodrigues and M.~Shillor,
{\em Stationary solutions to the thermistor problem},
J.~, 174:573--588, (1993).

\bibitem{JK82} D.~Jerison and C.E.~Kenig,
{\em The {N}eumann problem on {L}ipschitz domains},
Bull.~Amer.~Math.~Soc., 4:203--207, (1982).

\bibitem{JK95} D.~Jerison and C.E.~Kenig,
{\em The inhomogenuous {D}irichlet problem in Lipschitz domains},
J.~Funct.~Anal., 130:161--219, (1995).

\bibitem{JW84} A.~Johnsson and H.~Wallin,
{\em Function spaces on subsets of $\R^n$},
Mathematical Report, 2(1), Harwood Academic Publishers, (1984).

\bibitem{JJ99} J.~Jost, X-L.~Jost,
{\em Calculus of Variations},
CUP, 1999.

\bibitem{KSF09} K.L.~Kuttler, M.~Shillor and J.R.~Fern{\'a}ndez,
{\em Existence for the thermoviscoelastic thermistor problem},
Differ.~Equ.~Dyn.~Syst. 17(3):217--233, (2009).

\bibitem{MeSa11} J.M.~Melenk and S.~Sauter,
{\em Convergence analysis for finite element discretizations of the Helmholtz equation with Dirichlet-to-Neumann boundary conditions},
Math. Comp., 79:1871--1914, (2010).

\bibitem{MM07} I.~Mitrea and M.~Mitrea,
{\em The {P}oisson problem with mixed boundary conditions in {S}obolev and {B}esov spaces in non-smooth domains},
Trans.~Amer.~Math.~Soc., 359(9):4143--4182, (2007).

\bibitem{S97} G.~Savar\'e,
{\em Regularity and perturbation results for mixed second order elliptic problems},
Comm.~Partial Diff.~Eqns., 22:869--899, (1997).

\bibitem{ScZh90} L.~R.~Scott and S.~Zhang,
{\em Finite element interpolation of nonsmooth functions satisfying boundary conditions},
Math.~Comp., 54:481--493, (1990).

\bibitem{YL94} G.W.~Yuan and Z.H.~Liu,
{\em Existence and uniqueness of the $C^\alpha$ solution for the thermistor problem with mixed boundary value},
SIAM J.~Math.~Anal.,  25(4):1157--1166, (1994).

\bibitem{Zeidler} E. Zeidler,
{\em Nonlinear functional analysis and its applications},
Springer-Verlag, New York, (1990).

\bibitem{ZH10} J.~Zhang,
{\em On the thermistor problem with mixed boundary conditions},
Int.~J.~Pure Appl.~Math., 63(3):327--340, (2010).

\bibitem{ZL05} J.~Zhu and A.~F.~D Loula,
{\em Mixed finite element analysis of a thermally nonlinear coupled problem},
Numer.~Methods Partial Differential Equations, 22(1):180--196, (2006).

\bibitem{ZYL11} J.~Zhu, X.~Yu, and A.~F.~D.~Loula,
{\em Mixed discontinuous Galerkin analysis of thermallynonlinear coupled problem},
Comput.~Methods Appl.~Mech.~Engrg., 200:1479--1489, (2011).

\end{thebibliography}
\end{document}